\documentclass[12pt]{amsart}
\usepackage{amssymb,latexsym,amsmath,amscd,amsthm,graphicx, color}
\usepackage[all]{xy}
\usepackage{pgf,tikz}
\usepackage{mathrsfs}
\usetikzlibrary{arrows}
\definecolor{uuuuuu}{rgb}{0.26666666666666666,0.26666666666666666,0.26666666666666666}
\definecolor{xdxdff}{rgb}{0.49019607843137253,0.49019607843137253,1.}
\definecolor{ffqqqq}{rgb}{1.,0.,0.}

\definecolor{uuuuuu}{rgb}{0.26666666666666666,0.26666666666666666,0.26666666666666666}
\definecolor{qqwuqq}{rgb}{0.,0.39215686274509803,0.}
\definecolor{zzttqq}{rgb}{0.6,0.2,0.}
\definecolor{xdxdff}{rgb}{0.49019607843137253,0.49019607843137253,1.}
\definecolor{qqqqff}{rgb}{0.,0.,1.}
\definecolor{cqcqcq}{rgb}{0.7529411764705882,0.7529411764705882,0.7529411764705882}
\definecolor{sqsqsq}{rgb}{0.12549019607843137,0.12549019607843137,0.12549019607843137}

\theoremstyle{plain}

\newtheorem{theorem}[subsubsection]{Theorem}

\newtheorem{lemma}[subsubsection]{Lemma}

\newtheorem{prop}[subsubsection]{Proposition}

\theoremstyle{definition}

\newtheorem{prop1}[subsection]{Proposition}

\newtheorem{remark}[subsubsection]{Remark}
\newtheorem{rem}[subsection]{Remark}

\newcommand{\uu}{\cup}

\newcommand{\sci}{\subset}

\newcommand{\set}[1]{\{#1\}}

\newcommand{\ga}{\alpha}
\newcommand{\gb}{\beta}

\newcommand{\gd}{\delta}
\renewcommand{\gg}{\gamma}

\newcommand{\gq}{\theta}

\newcommand{\tit}{\textit}

\newcommand{\D}[1]{\mathbb{#1}}

\newcommand{\te}{\text}

\newcommand{\ol}{\overline}
\newcommand{\ul}{\underline}

\newcommand{\pa}{\partial}

\begin{document}
To appear, Communications of the Korean Mathematical Society
\title{Uniform distributions on curves and quantization}

\author{Joseph Rosenblatt}
\address{Department of Mathematics \\
University of Illinois\\
1409 W. Green Street \\
Urbana, Illinois 61801-2907, USA}
\email{rosnbltt@illinois.edu}

 \author{Mrinal Kanti Roychowdhury}
\address{School of Mathematical and Statistical Sciences\\
University of Texas Rio Grande Valley\\
1201 West University Drive\\
Edinburg, TX 78539-2999, USA.}
\email{mrinal.roychowdhury@utrgv.edu}


\date{}
\maketitle

\pagestyle{myheadings}\markboth{Joseph Rosenblatt and Mrinal Kanti Roychowdhury}{Uniform distributions on curves and quantization}

\begin{abstract}
The basic goal of quantization for probability distribution is to reduce the number of values, which is typically uncountable, describing a probability distribution to some finite set and thus to make an approximation of a continuous probability distribution by a discrete distribution. It has broad application in signal processing and data compression. In this paper, first we define the uniform distributions on different curves such as a line segment, a circle, and the boundary of an equilateral triangle. Then, we give the exact formulas to determine the optimal sets of $n$-means and the $n$th quantization errors for different values of $n$ with respect to the uniform distributions defined on the curves. In each case, we further calculate the quantization dimension and show that it is equal to the dimension of the object; and the quantization coefficient exists as a finite positive number.  This supports the well-known result of  Bucklew and Wise (1982), which says that for a Borel probability measure $P$ with non-vanishing absolutely continuous part the quantization coefficient exists as a finite positive number.
\end{abstract}

\section{Introduction}
`Quantization' refers to the process of approximating the continuous set of values in the image data with a finite (preferably small) set of values
 with broad application in engineering and technology (see \cite{GG, GN, Z}). Although the work of quantization in engineering science has a long history, rigorous mathematical treatment was given by Graf and Luschgy (see \cite{GL1}).
The quantization for a probability distribution refers to the idea of estimating a given probability by a discrete probability with a given number $n$ of supporting points.

Let $P$ denote a Borel probability measure on $\D R^d$, where $d\geq 1$.
For a finite set $\ga\sci\D R^d$, the error $\int \min_{a \in \ga} \|x-a\|^2 dP(x)$ is often referred to as the \tit{cost} or \tit{distortion error} for $\ga$, and is denoted by $V(P; \ga)$. For any positive integer $n$, write $V_n:=V_n(P)=\inf\set{V(P; \ga) :\alpha \subset \mathbb R^d, \text{ card}(\alpha) \leq n}$. Then, $V_n$ is called the $n$th quantization error for $P$. We assume that $\int\|x\|^2 dP(x)<\infty$ to make sure that there is a set $\ga$ for which the infimum occurs (see \cite{AW, GKL, GL1, GL2}). Such a set $\ga$ for which the infimum occurs and contains no more than $n$-points is called an \tit{optimal set of $n$-means} and the elements of an optimal set are called \tit{optimal quantizers}. In some literature it is also referred to as \tit{principal points} (see \cite{MKT}). The numbers
\[\ul D(P):=\liminf_{n\to \infty}  \frac{2\log n}{-\log V_n(P)}, \te{ and } \ol D(P):=\limsup_{n\to \infty} \frac{2\log n}{-\log V_n(P)} \]
are, respectively, called the \tit{lower} and \tit{upper quantization dimensions} of the probability measure $P$. If $\ul D(P)=\ol D (P)$, the common value is called the \tit{quantization dimension} of $P$ and is denoted by $D(P)$. Quantization dimension measures the speed at which the specified measure of the error tends to zero as $n$ approaches to infinity.
For any $s>0$, the two numbers $\liminf_{n\to \infty} n^{ 2/s} V_n(P)$ and $\limsup_{n\to \infty} n^{  2/s} V_n(P)$ are, respectively, called the $s$-dimensional \tit{lower} and \tit{upper quantization coefficients} for $P$. If the $s$-dimensional \tit{lower} and \tit{upper quantization coefficients} are equal, we call it the $s$-dimensional quantization coefficient for $P$. If the $s$-dimensional lower and the upper quantization coefficients are finite and positive, then $s$ equals the quantization dimension of $P$ (see \cite{GL1}). It is well-known that for a Borel probability measure $P$ with non-vanishing absolutely continuous part $\lim_{n\to \infty} n^{2/d} V_n(P)$ is finite and strictly positive  (see  \cite{BW}). This implies that the quantization dimension of a Borel probability measure with non-vanishing absolutely continuous part equals the dimension $d$ of the underlying space.
For a finite set $\ga\sci \D R^d$, the \tit{Voronoi region} generated by $a\in \ga$, denoted by $M(a|\ga)$, is defined to be the set of all elements in $\D R^d$ which are nearest to $a$. The set $\set{M(a|\ga) : a \in \ga}$ is called the \tit{Voronoi diagram} or \tit{Voronoi tessellation} of $\D R^d$ with respect to $\ga$. The point $a$ is called the centroid of its own Voronoi region if $a=E(X : X\in M(a|\ga))$, where $X$ is a $P$-distributed random variable.
Let us now state the following proposition (see \cite{GG, GL1}).
\begin{prop1} \label{prop0}
Let $\ga$ be an optimal set of $n$-means, $a \in \ga$, and $M (a|\ga)$ be the Voronoi region generated by $a\in \ga$. Then, for every $a \in\ga$,
$(i)$ $P(M(a|\ga))>0$, $(ii)$ $ P(\partial M(a|\ga))=0$, $(iii)$ $a=E(X : X \in M(a|\ga))$, and $(iv)$ $P$-almost surely the set $\set{M(a|\ga) : a \in \ga}$ forms a Voronoi partition of $\D R^d$.
\end{prop1}
Proposition~\ref{prop0} says that if $\ga$ is an optimal set and $a\in\ga$, then $a$ is the \tit{conditional expectation} of the random variable $X$ given that $X$ takes values in the Voronoi region of $a$. Recently, optimal quantization for uniform distributions on different regions have been investigated by several authors, for example, see \cite{DR, R, RR}.

In this paper, there are three subsections. In Subsection~\ref{sec0}, Subsection~\ref{sec1}, and Subsection~\ref{sec2}, first we have defined the uniform distributions on a line segment, on a unit circle, and on the boundary of an equilateral triangle. Then, in Theorem~\ref{theorem0}, Theorem~\ref{theorem1}, and Theorem~\ref{Th41}, we give the exact formulas to determine the optimal sets of $n$-means and the corresponding quantization errors for different values of $n\in \D N$. We have further shown that the quantization dimension of the uniform distribution in each case is one, which equals the dimension of the line, circle, and the boundary of the equilateral triangle. Moreover, we have shown that in each case the quantization coefficient exists as a finite positive number. As described in Remark~\ref{rem1}, once the optimal sets of $n$-means are known for a line segment, a circle, and an equilateral triangle, by giving an affine transformation, one can easily obtain them for any line segment, any circle, and the boundary of any equilateral triangle. We did not find any literature where the optimal quantization for a probability distributions on a curve has been investigated. Thus, our work in this paper can be considered as a first advance in this direction. Finally, we would like to mention that the technique used in this paper will be of great help to investigate the optimal quantization for any distribution on any curve.

\section{Main Result}
In this section, in three subsections, we give the main results of the paper.
Let $i$ and $j$ be the unit vectors in the positive directions of $x_1$- and $x_2$-axes, respectively. By the position vector $\tilde a$ of a point $A$, it is meant that $\overrightarrow{OA}=\tilde a$. In the sequel, we will identify the position vector of a point $(a_1, a_2)$ by $(a_1, a_2):=a_1 i +a_2 j$, and apologize for any abuse in notation. For any two vectors $\vec u$ and $\vec v$, let $\vec u \cdot \vec v$ denote the dot product between the two vectors $\vec u$ and $\vec v$. Then, for any vector $\vec v$, by $(\vec v)^2$, we mean $(\vec v)^2:= \vec v\cdot \vec v$. Thus, $|\vec v|:=\sqrt{\vec v\cdot \vec v}$, which is called the length of the vector $\vec v$. For any two position vectors $\tilde a:=( a_1, a_2)$ and $\tilde b:=( b_1, b_2)$, we write $\rho(\tilde a, \tilde b):=\|(a_1, a_2)-(b_1, b_2)\|^2=(a_1-b_1)^2 +(a_2-b_2)^2$, which gives the squared Euclidean distance between the two points $(a_1, a_2)$ and $(b_1, b_2)$. By $E(X)$ and $V:=V(X)$, we represent the expectation and the variance of a random variable $X$ with respect to the probability distribution under consideration. Let $P$ and $Q$ with position vectors $\tilde p$ and $\tilde q$ belong to an optimal set of $n$-means for some positive integer $n$, and let $D$ with position vector $\tilde d$ be a point on the boundary of the Voronoi regions of the points $P$ and $Q$. Since the boundary of the Voronoi regions of any two points is the perpendicular bisector of the line segment joining the points, we have
$|\overrightarrow{DP}|=|\overrightarrow{DQ}|, \te{ i.e., } (\overrightarrow{DP})^2=(\overrightarrow{DQ})^2$ implying
$(\tilde p-\tilde d)^2=(\tilde q-\tilde d)^2$, i.e., $\rho(\tilde p, \tilde d)=\rho(\tilde q, \tilde d)$. In the sequel, such an equation will be referred to as a \tit{canonical equation}. In the paper, there are some decimal numbers; they are rational approximations of some real numbers.

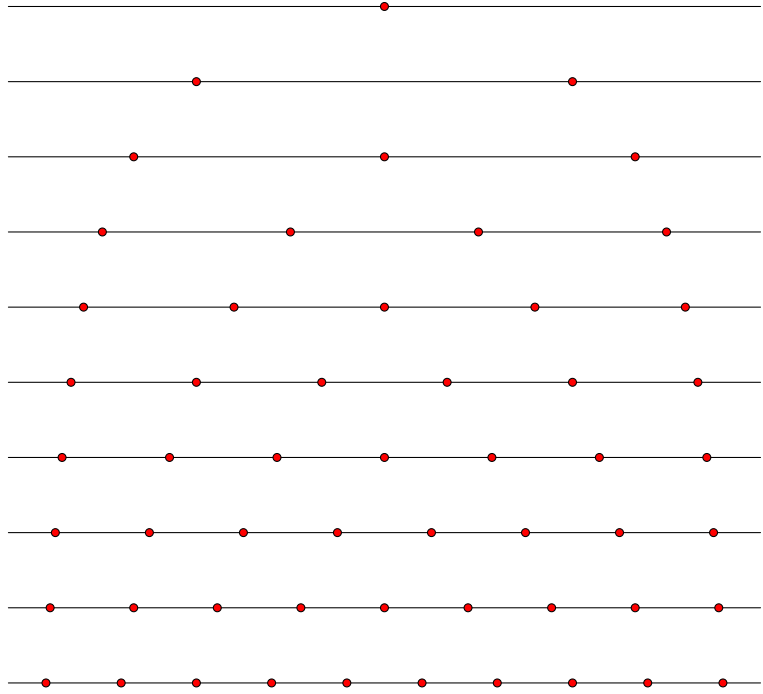
\begin{figure}
\begin{tikzpicture}[line cap=round,line join=round,>=triangle 45,x=1.0cm,y=1.0cm]
\clip(-0.382651326233836,1.6009686885066525) rectangle (10.68280344005557,11.74203853307883);
\draw (0.,10.)-- (10.,10.);
\draw (0.,9.)-- (10.,9.);
\draw (0.,8.)-- (10.,8.);
\draw (0.,7.)-- (10.,7.);
\draw (0.,6.)-- (10.,6.);
\draw (0.,5.)-- (10.,5.);
\draw (0.,4.)-- (10.,4.);
\draw (0.,3.)-- (10.,3.);
\draw (0.,2.)-- (10.,2.);
\draw (0.,11.)-- (10.,11.);
\begin{scriptsize}
\draw [fill=ffqqqq] (0.8333333333333334,6.) circle (1.5pt);
\draw [fill=ffqqqq] (2.5,6.) circle (1.5pt);
\draw [fill=ffqqqq] (4.166666666666667,6.) circle (1.5pt);
\draw [fill=ffqqqq] (5.833333333333333,6.) circle (1.5pt);
\draw [fill=ffqqqq] (7.5,6.) circle (1.5pt);
\draw [fill=ffqqqq] (9.166666666666666,6.) circle (1.5pt);
\draw [fill=ffqqqq] (0.7142857142857143,5.) circle (1.5pt);
\draw [fill=ffqqqq] (2.142857142857143,5.) circle (1.5pt);
\draw [fill=ffqqqq] (3.5714285714285716,5.) circle (1.5pt);
\draw [fill=ffqqqq] (5.,5.) circle (1.5pt);
\draw [fill=ffqqqq] (6.428571428571429,5.) circle (1.5pt);
\draw [fill=ffqqqq] (7.857142857142857,5.) circle (1.5pt);
\draw [fill=ffqqqq] (9.285714285714286,5.) circle (1.5pt);
\draw [fill=ffqqqq] (0.625,4.) circle (1.5pt);
\draw [fill=ffqqqq] (1.875,4.) circle (1.5pt);
\draw [fill=ffqqqq] (3.125,4.) circle (1.5pt);
\draw [fill=ffqqqq] (4.375,4.) circle (1.5pt);
\draw [fill=ffqqqq] (5.625,4.) circle (1.5pt);
\draw [fill=ffqqqq] (6.875,4.) circle (1.5pt);
\draw [fill=ffqqqq] (8.125,4.) circle (1.5pt);
\draw [fill=ffqqqq] (9.375,4.) circle (1.5pt);
\draw [fill=ffqqqq] (0.5555555555555556,3.) circle (1.5pt);
\draw [fill=ffqqqq] (1.6666666666666667,3.) circle (1.5pt);
\draw [fill=ffqqqq] (2.7777777777777777,3.) circle (1.5pt);
\draw [fill=ffqqqq] (3.888888888888889,3.) circle (1.5pt);
\draw [fill=ffqqqq] (5.,3.) circle (1.5pt);
\draw [fill=ffqqqq] (6.111111111111111,3.) circle (1.5pt);
\draw [fill=ffqqqq] (7.222222222222222,3.) circle (1.5pt);
\draw [fill=ffqqqq] (8.333333333333334,3.) circle (1.5pt);
\draw [fill=ffqqqq] (9.444444444444445,3.) circle (1.5pt);
\draw [fill=ffqqqq] (0.5,2.) circle (1.5pt);
\draw [fill=ffqqqq] (1.5,2.) circle (1.5pt);
\draw [fill=ffqqqq] (2.5,2.) circle (1.5pt);
\draw [fill=ffqqqq] (3.5,2.) circle (1.5pt);
\draw [fill=ffqqqq] (4.5,2.) circle (1.5pt);
\draw [fill=ffqqqq] (5.5,2.) circle (1.5pt);
\draw [fill=ffqqqq] (6.5,2.) circle (1.5pt);
\draw [fill=ffqqqq] (7.5,2.) circle (1.5pt);
\draw [fill=ffqqqq] (8.5,2.) circle (1.5pt);
\draw [fill=ffqqqq] (9.5,2.) circle (1.5pt);
\draw [fill=ffqqqq] (1.,7.) circle (1.5pt);
\draw [fill=ffqqqq] (3.,7.) circle (1.5pt);
\draw [fill=ffqqqq] (5.,7.) circle (1.5pt);
\draw [fill=ffqqqq] (7.,7.) circle (1.5pt);
\draw [fill=ffqqqq] (9.,7.) circle (1.5pt);
\draw [fill=ffqqqq] (1.25,8.) circle (1.5pt);
\draw [fill=ffqqqq] (3.75,8.) circle (1.5pt);
\draw [fill=ffqqqq] (6.25,8.) circle (1.5pt);
\draw [fill=ffqqqq] (8.75,8.) circle (1.5pt);
\draw [fill=ffqqqq] (1.6666666666666667,9.) circle (1.5pt);
\draw [fill=ffqqqq] (5.,9.) circle (1.5pt);
\draw [fill=ffqqqq] (8.333333333333334,9.) circle (1.5pt);
\draw [fill=ffqqqq] (2.5,10.) circle (1.5pt);
\draw [fill=ffqqqq] (7.5,10.) circle (1.5pt);
\draw [fill=ffqqqq] (5.,11.) circle (1.5pt);
\end{scriptsize}
\end{tikzpicture}
\caption{Optimal configuration of $n$ means with respect to a uniform distribution on a line segment for $1\leq n\leq 10$.} \label{Fig0}
\end{figure}

 \subsection{Optimal quantization for a uniform distribution on a line segment} \label{sec0}
Without any loss of generality we can assume the line segment as a closed interval $[a, b]$, where $0<a<b<+\infty$. Let $P$ be the uniform distribution defined on the closed interval $[a, b]$. Let the probability density function for $P$ is given by $f$. Then, we have
\[f(x)=\left\{\begin{array}{cc}
\frac 1{b-a} &\te{ if } x\in [a, b],\\
0 & \te{otherwise},
\end{array}
\right.\]
implying $dP(x)=P(dx)=f(x)dx=\frac 1{b-a} dx$.

The following theorem gives the optimal sets of $n$-means for all $n\in \D N$.

\begin{theorem} \label{theorem0}   Let $P$ be the uniform distribution on the closed interval $[a, b]$. Then, the optimal set $n$-means is given by $ \ga_n:=\set{a+\frac {2i-1}{2n}(b-a) : 1\leq i\leq n}$, and the corresponding quantization error is
$V_n:=V_n(P)=\frac{(a-b)^2}{12 n^2}.$
\end{theorem}

\begin{proof} Let $\ga_n:=\set{a_1, a_2, \cdots, a_n}$ be an optimal set of $n$-means. Since the points in an optimal set are the conditional expectations in their own Voronoi regions, without any loss of generality, we can assume that $a<a_1<a_2<\cdots<a_n<b$. Then, by Proposition~\ref{prop0}, we have
\begin{align}
a_1&=E(X : X \in [a, \frac{1}{2} (a_1+a_2)]), \label{eq01} \\
a_i&=E(X : X \in [\frac{a_{i-1}+a_i}{2}, \frac{a_i+a_{i+1}}{2}]) \te{ for } 2\leq i\leq n-1,  \label{eq21}\\
a_n&=E(X : X\in [\frac{a_{n-1}+a_{n}}{2}, b]). \label{eq3}
\end{align}
By \eqref{eq01}, we have
\[a_1=\frac{\int_{a}^{ \frac{1}{2} (a_1+a_2)} x dP}{\int_{a}^{\frac{1}{2} (a_1+a_2)} dP}=\frac{\int_{a}^{\frac{1}{2} (a_1+a_2)} x f(x) dx}{\int_{a}^{ \frac{1}{2} (a_1+a_2)} f(x) dx}=\frac{1}{4} \left(2 a+a_1+a_2\right)\]
implying
\begin{equation} \label{eq4} 2(a_1-a)=a_2-a_1.\end{equation}
Similarly, by \eqref{eq21} and \eqref{eq3}, we have
\begin{equation} \label{eq5} a_{i-1}-a_i=a_{i+1}-a_i \te{ for } 2\leq i\leq n-1, \te{ and } a_n-a_{n-1}=2(b-a_n).\end{equation}
Combining the equations in \eqref{eq4} and \eqref{eq5}, we have
\begin{equation}\label{eq6}
2(a_1-a)=a_i-a_{i-1}=2(b-a_n)=k \te{ (say) for } 2\leq i\leq n
\end{equation}
implying \[(a_1-a)+\sum_{i=2}^{n}(a_i-a_{i-1})+(b-a_n)=\frac k 2+(n-1) k+\frac k2, \te{ i.e., } b-a=nk, \te{ i.e., }  k=\frac{b-a}n.\]
Thus, putting $k=\frac{b-a}n$, by \eqref{eq6}, we have
\[a_1=a+\frac 1{2n}(b-a), \, a_2=a+\frac 3{2n}(b-a), \, a_3=a+\frac 5{2n}(b-a), \cdots, a_n=a+\frac {2n-1} {2n} (b-a)\]
yielding the fact that $\ga_n:=\set{a+\frac {2i-1}{2n}(b-a) : 1\leq i\leq n}$ forms an optimal set of $n$-means for $P$ (see Figure~\ref{Fig0}). Notice that the probability density function is constant, and the Voronoi regions of the points $a_i$ for $1\leq i\leq n$ are of equal lengths. This yields the fact that the distortion errors due to each $a_i$ are equal. Hence, the $n$th quantization error $V_n:=V_n(P)$ is given by
\begin{align*}
V_n&=\int \min_{a\in \ga_n} (x-a)^2dP=\frac{n}{b-a} \int_a^{\frac 12(a_1+a_2)}  (x-a_1)^2  dx=\frac{(a-b)^2}{12 n^2}.
\end{align*}
 Thus, the proof of the theorem is complete.
\end{proof}

\begin{remark}
 By Theorem~\ref{theorem0}, we see that if $V_n$ is the quantization error of $n$-means for the uniform distribution on $[a, b]$, then $V_n=\frac{(a-b)^2}{12 n^2}$. Thus,
\[D(P)=\lim_{n\to \infty} \frac{2 \log n}{-\log V_n}=\lim_{n\to \infty} \frac{2 \log n}{-\log \frac{(a-b)^2}{12 n^2}}=1,\]
which equals the dimension of the line segment.
Moreover, we have
$\mathop{\lim}\limits_{n\to \infty} n^2 V_n=\frac{(a-b)^2}{12},$
which is a finite positive number.
\end{remark}

\begin{figure}
\begin{tikzpicture}[line cap=round,line join=round,>=triangle 45,x=1.0cm,y=1.0cm]
\clip(-1.9724657463944237,-1.225618187448618) rectangle (10.65632482539858,4.478240453944644);
\draw(0.,3.) circle (1.cm);
\draw(3.,3.) circle (1.cm);
\draw [dash pattern=on 1pt off 1pt,color=ffqqqq] (3.,3.) circle (0.6366197723675814cm);
\draw(6.,3.) circle (1.cm);
\draw [dash pattern=on 1pt off 1pt,color=ffqqqq] (6.000004929463358,3.) circle (0.8269949294633572cm);
\draw(9.,3.) circle (1.cm);
\draw [line width=1.2pt,dash pattern=on 1pt off 1pt,color=ffqqqq] (9.,3.) circle (0.9003166380779603cm);
\draw [line width=1.2pt] (0.,0.) circle (1.cm);
\draw [line width=1.2pt,dash pattern=on 1pt off 1pt,color=ffqqqq] (4.755122099159953E-7,0.) circle (0.9354892334266961cm);
\draw [line width=1.2pt] (3.,0.) circle (1.cm);
\draw [line width=1.2pt,dash pattern=on 1pt off 1pt,color=ffqqqq] (3.,-5.7904639672247756E-6) circle (0.9549238680874043cm);
\draw [line width=1.2pt] (6.,0.) circle (1.cm);
\draw [line width=1.2pt] (9.,0.) circle (1.cm);
\draw [line width=1.2pt,dash pattern=on 1pt off 1pt,color=ffqqqq] (6.0000015565930465,-1.2604320313800195E-6) circle (0.9667686661933416cm);
\draw [line width=1.2pt,dash pattern=on 1pt off 1pt,color=ffqqqq] (9.,-8.949751892969526E-6) circle (0.9745020993028948cm);
\begin{scriptsize}
\draw [fill=ffqqqq] (0.,3.) circle (0.5pt);
\draw [fill=ffqqqq] (3.,3.6366197723675815) circle (1.0pt);
\draw [fill=ffqqqq] (3.,2.3633802276324185) circle (1.0pt);
\draw [fill=ffqqqq] (6.4135,3.7162) circle (1.0pt);
\draw [fill=ffqqqq] (5.17301,3.) circle (1.0pt);
\draw [fill=ffqqqq] (6.4135,2.2838) circle (1.0pt);
\draw [fill=ffqqqq] (9.63662,3.63662) circle (1.0pt);
\draw [fill=ffqqqq] (8.36338,3.63662) circle (1.0pt);
\draw [fill=ffqqqq] (8.36338,2.36338) circle (1.0pt);
\draw [fill=ffqqqq] (9.63662,2.36338) circle (1.0pt);
\draw [fill=ffqqqq] (0.756827,0.549867) circle (1.0pt);
\draw [fill=ffqqqq] (-0.289082,0.889703) circle (1.0pt);
\draw [fill=ffqqqq] (-0.935489,0.) circle (1.0pt);
\draw [fill=ffqqqq] (-0.289082,-0.889703) circle (1.0pt);
\draw [fill=ffqqqq] (0.756827,-0.549867) circle (1.0pt);
\draw [fill=ffqqqq] (3.82699,0.477464829275686) circle (1.0pt);
\draw [fill=ffqqqq] (3.,0.954929658551372) circle (1.0pt);
\draw [fill=ffqqqq] (2.17301,0.477464829275686) circle (1.0pt);
\draw [fill=ffqqqq] (2.17301,-0.477464829275686) circle (1.0pt);
\draw [fill=ffqqqq] (3.,-0.954929658551372) circle (1.0pt);
\draw [fill=ffqqqq] (3.82699,-0.477464829275686) circle (1.0pt);
\draw [fill=ffqqqq] (6.87103,0.419464) circle (1.0pt);
\draw [fill=ffqqqq] (6.21513,0.942528) circle (1.0pt);
\draw [fill=ffqqqq] (5.39723,0.755848) circle (1.0pt);
\draw [fill=ffqqqq] (5.03323,0.) circle (1.0pt);
\draw [fill=ffqqqq] (5.39723,-0.755848) circle (1.0pt);
\draw [fill=ffqqqq] (6.21513,-0.942528) circle (1.0pt);
\draw [fill=ffqqqq] (6.87103,-0.419464) circle (1.0pt);
\draw [fill=ffqqqq] (9.90032,0.372923) circle (1.0pt);
\draw [fill=ffqqqq] (9.37292,0.900316) circle (1.0pt);
\draw [fill=ffqqqq] (8.62708,0.900316) circle (1.0pt);
\draw [fill=ffqqqq] (8.09968,0.372923) circle (1.0pt);
\draw [fill=ffqqqq] (8.09968,-0.372923) circle (1.0pt);
\draw [fill=ffqqqq] (8.62708,-0.900316) circle (1.0pt);
\draw [fill=ffqqqq] (9.37292,-0.900316) circle (1.0pt);
\draw [fill=ffqqqq] (9.90032,-0.372923) circle (1.0pt);
\end{scriptsize}
\end{tikzpicture}
\caption{Optimal configuration of $n$ means with respect to a uniform distribution on a circle for $1\leq n\leq 8$.} \label{Fig00}
\end{figure}
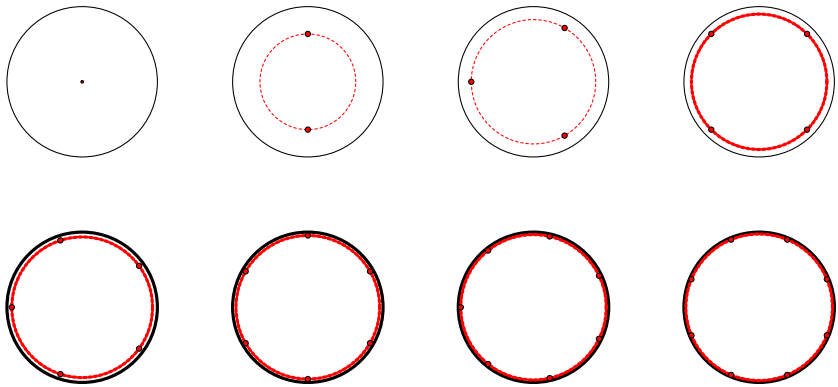
\subsection{Optimal quantization for a uniform distribution on a unit circle} \label{sec1}
Let $L$ be the unit circle given by the parametric equations:
$L:=\set{(x_1, x_2) :  x_1=\cos \gq, \, x_2=\sin\gq \te{ for } 0\leq \gq\leq 2\pi}.$
Let the positive direction of the $x_1$-axis cut the circle at the point $A$, i.e., $A$ is represented by the parametric value $\gq=0$. Let $s$ be the distance of a point on $L$ along the arc starting from the point $A$ in the counterclockwise direction. Then,
\[ds=\sqrt{\Big(\frac {dx_1}{d\gq}\Big)^2+\Big(\frac{dx_2}{d\gq}\Big)^2}\,d\gq=d\gq.\]
 Then, the  probability density function (pdf) $f(x_1, x_2)$ for $P$ is given by
\[f(x_1, x_2)=\left\{\begin{array}{ccc}
\frac 1 {2\pi} & \te{ if } (x_1, x_2) \in L,\\
 0  & \te{ otherwise}.
\end{array}\right.
\]
Thus, we have $dP(s)=P(ds)=f(x_1, x_2) ds=\frac 1{2\pi} d\gq$.
Moreover, we know that if $\hat \gq$ radians is the
central angle subtended by an arc of length $S$ of the unit circle, then $S =\hat \gq$, and
\[P(S)=\int_S dP(s)=\frac 1{2 \pi} \int_S d\gq= \frac{\hat\gq} {2\pi}.\]
The following theorem gives the optimal sets of $n$-means and the $n$th quantization errors for all $n\in \D N$.
\begin{theorem} \label{theorem1}
Let $\ga_n$ be an optimal set of $n$-means for the uniform distribution $P$ on the unit circle $x_1^2+x_2^2=1$ for $n\in\D N$. Then,
\[\ga_n:=\Big\{ \frac n{ \pi} \Big(\sin (\frac \pi n) \cdot \cos ((2j-1){\frac \pi n}), \   \sin (\frac \pi n) \cdot \sin ((2j-1){\frac \pi n})\Big) : j=1, 2, \cdots, n \Big  \} \]
forms an optimal set of $n$-means, and the corresponding quantization error is given by $1-\frac{n^2}{\pi^2}  \sin^2\left(\frac{\pi }{n}\right).$
\end{theorem}

\begin{proof}
Let $\ga_n:=\set{a_1, a_2, \cdots, a_n}$ be an optimal set of $n$-means for $P$ for $n\in \D N$. Let $\gq_1-\gq_0, \, \gq_2-\gq_1, \,  \cdots, \, \gq_n-\gq_{n-1}$ be the angles in radians subtended by the Voronoi regions of the points $a_1, a_2, \cdots, a_n$, and let the corresponding lengths of the arcs be $S_1, S_2, \cdots, S_n$, respectively, where $\gq_0=0$ and $\gq_n=2\pi$. Then, $S_k=\gq_k-\gq_{k-1}$ for $1\leq k\leq n$. Thus, for $1\leq k\leq n$, we have
\begin{align*}
a_k&=E(X : X \in S_k)=\frac 1{P(S_k)} \int_{S_k}(i \cos \gq + j \sin \gq) dP=\frac {2 \pi}  {\gq_k-\gq_{k-1}} \int_{\gq_{k-1}}^{\gq_k} \frac 1 {2 \pi}(i \cos \gq +j \sin \gq) d\gq,
\end{align*}
which after simplification implies
\begin{align} \label{eq00} a_k=\frac 1 {\gq_k-\gq_{k-1}} \left(\sin \gq_{k}-\sin\gq_{k-1}, \, \cos\gq_{k-1}-\cos \gq_k\right).\end{align}
The $n$th quantization error is given by
\begin{align*}
V_n&=\int \min_{a\in \ga_n}\|(\cos \gq, \sin \gq)-a\|^2 dP=\frac 1 {2 \pi} \sum_{k=1}^n \int_{\gq_{k-1}}^{\gq_k}\|(\cos \gq, \sin \gq)-a_k\|^2 d\gq \\
&=\frac 1 {2 \pi} \sum_{k=1}^n \int_{\gq_{k-1}}^{\gq_k} \Big(\Big(\cos \gq-\frac{\sin \theta _k-\sin \theta _{k-1}}{\theta _{k}- \theta _{k-1}}\Big){}^2+\Big(\sin \gq-\frac{\cos \theta _{k-1}-\cos \theta _{k}}{\theta _{k}-\theta _{k-1}}\Big)^2\Big)d\gq\\
&=\frac 1 {2 \pi} \sum_{k=1}^n \Big(\theta _{k}-\theta _{k-1}+2\frac {\left(\cos \left(\theta _{k-1}-\theta _k\right)-1\right)}{ \left(\theta _k-\theta _{k-1}\right)}\Big)=\frac 1 {2 \pi} \sum_{k=1}^n \Big(\theta _{k}-\theta _{k-1}-4 \frac{\sin^2 \frac{\gq_k-\gq_{k-1}}{2}}{\gq_k-\gq_{k-1}}\Big)\\
&=1-\frac{2}{\pi}\sum_{k=1}^n \frac{\sin^2 \frac{\gq_k-\gq_{k-1}}{2}}{\gq_k-\gq_{k-1}}.
\end{align*}
The distortion error $V_n$ being optimal, we must have $\frac {\pa V_n}{\pa \gq_k}=0$ for all $k=1, 2, \cdots, n-1$, yielding the fact that
\[-\frac {\sin^2\frac{\gq_k-\gq_{k-1}}{2}}{(\gq_k-\gq_{k-1})^2}+\frac{\sin(\gq_k-\gq_{k-1})}{ (\gq_k-\gq_{k-1})}+\frac {\sin^2\frac{\gq_{k+1}-\gq_{k}}{2}}{(\gq_{k+1}-\gq_{k})^2}-\frac{\sin(\gq_{k+1}-\gq_{k})}{ (\gq_{k+1}-\gq_{k})}=0 \]
implying
\begin{equation} \label{eq1} \frac {\sin^2\frac{\gq_k-\gq_{k-1}}{2}}{(\gq_k-\gq_{k-1})^2}-\frac{\sin(\gq_k-\gq_{k-1})}{ (\gq_k-\gq_{k-1})}=\frac {\sin^2\frac{\gq_{k+1}-\gq_{k}}{2}}{(\gq_{k+1}-\gq_{k})^2}-\frac{\sin(\gq_{k+1}-\gq_{k})}{ (\gq_{k+1}-\gq_{k})}\end{equation}
for all $k=1, 2, \cdots, n-1$. From the above recurrence relation, we have
\begin{equation} \label{eq2} \frac {\sin^2\frac{\gq_1-\gq_{0}}{2}}{(\gq_1-\gq_{0})^2}-\frac{\sin(\gq_1-\gq_{0})}{ (\gq_1-\gq_{0})}=\frac {\sin^2\frac{\gq_{k}-\gq_{k-1}}{2}}{(\gq_{k}-\gq_{k-1})^2}-\frac{\sin(\gq_{k}-\gq_{k-1})}{ (\gq_{k}-\gq_{k-1})}.\end{equation}
for $k=2, 3, \cdots, n$. This yields the fact that
\[\gq_1-\gq_0=\gq_2-\gq_1=\gq_3-\gq_2=\cdots=\gq_n-\gq_{n-1}=\frac {2 \pi}{n}.\]
Thus, we have $\gq_k =\frac {2 \pi k} n$ for $k=1,2, \cdots, n$. Hence, by \eqref{eq00}, we deduce that if $\ga_n:=\set{a_1, a_2, \cdots, a_n}$ is an optimal set of $n$-means, then $a_k= \frac n{ \pi} \Big(\sin (\frac \pi n) \cdot \cos ((2k-1){\frac \pi n}), \   \sin (\frac \pi n) \cdot \sin ((2k-1){\frac \pi n})\Big)$ for $k=1, 2, \cdots, n$  (see Figure~\ref{Fig00}).
Notice that in the optimal set of $n$-means, the distortion errors contributed by each point in their own Voronoi regions are equal. Thus, the quantization error for $n$-means is given by
\begin{align*}
V_n&=n\Big(\te{quanization error due to the point $a_1$ in the opitmal set of $n$-means}\Big)\\
&=n \int_{0}^{\frac {2 \pi} {n}} \|(\cos\gq, \sin\gq)-(\frac n {2 \pi} \sin\frac{2 \pi} {n}, \frac n { \pi} \sin^2\frac{ \pi} {n})\|^2 dP\\
&=n \int_{0}^{\frac {2 \pi} {n}} \frac 1 {2 \pi} \|(\cos\gq, \sin\gq)-(\frac n {2 \pi} \sin\frac{2 \pi} {n}, \frac n { \pi} \sin^2\frac{ \pi} {n})\|^2 d\gq,
\end{align*}
which after simplification yields $V_n=1-\frac {n^2}{\pi^2} \sin^2\frac{ \pi} {n}.$ Thus, the proof of the theorem is complete.
\end{proof}

\begin{remark}
 By Theorem~\ref{theorem1}, we see that if $V_n$ is the quantization error of $n$-means for the uniform distribution on the circle, then $V_n=1-\frac {n^2}{\pi^2} \sin^2\frac{ \pi} {n}$. Thus,
\[D(P)=\lim_{n\to \infty} \frac{2 \log n}{-\log V_n}=\lim_{n\to \infty} \frac{2 \log n}{-\log \left(1-\frac {n^2}{\pi^2} \sin^2\frac{ \pi} {n}\right)}=1,\]
 which equals the dimension of the circle.
Moreover, we have $ \lim_{n\to \infty} n^2 V_n=\frac{\pi ^2}{3},$
which is a finite positive number.
\end{remark}

\begin{prop} \label{prop000}
The points in an optimal set of $n$-means for the uniform distribution $P$ on the unit circle $x_1^2+x_2^2=1$ form a concentric circle of radius $\frac {n}{\pi} \sin\frac{ \pi} {n}$.
\end{prop}
\begin{proof} By Theorem~\ref{theorem1}, we see that if $\ga_n$ is an optimal set of $n$-means for a positive integer $n$, then
$\ga_n=\set{a_1, a_2, \cdots, a_n}$, where
 $a_k= \frac n{ \pi} \Big(\sin (\frac \pi n) \cdot \cos ((2k-1){\frac \pi n}), \   \sin (\frac \pi n) \cdot \sin ((2k-1){\frac \pi n})\Big)$ for $k=1, 2, \cdots, n$. The distance of the points $a_k$ from the center $(0, 0)$ is $\frac {n}{\pi} \sin\frac{ \pi} {n}$ yielding the fact that the points $a_k$ form the circle $x_1^2+x_2^2=\frac {n^2}{\pi^2} \sin^2\frac{ \pi} {n}$, which is the proposition.
\end{proof}

\begin{remark}
By Proposition~\ref{prop000}, we know that the points in an optimal set of $n$-means for the uniform distribution $P$ on the unit circle $x_1^2+x_2^2=1$ lie on the circle $x_1^2+x_2^2=\frac {n^2}{\pi^2} \sin^2\frac{ \pi} {n}$. Notice that the radius of the circle $\frac {n}{\pi} \sin\frac{ \pi} {n}$ is an increasing function of $n$ and approaches to one as $n \to \infty$, i.e., the concentric circle formed by the points in an optimal set of $n$-means approaches to the circle  $x_1^2+x_2^2=1$ as $n\to \infty$ (see Figure~\ref{Fig00}).
 \end{remark}

 \subsection{Optimal quantization for a uniform distribution on the boundary of an equilateral triangle} \label{sec2}

Let $P$ be the uniform distribution defined on the boundary $L$ of the equilateral triangle with vertices $O(0, 0)$, $A(1, 0)$, and $B(\frac 1 2, \frac {\sqrt{3}} 2)$. Let $s$ represent the distance of any point on $L$ from the origin tracing along the boundary of the triangle in the counterclockwise direction. Then, the points $O$, $A$, $B$ are, respectively, represented by $s=0$, $s=1$, $s=2$, and the point with $s=3$ coincides with $O$. Then, the probability density function (pdf) $f(s)$ of the uniform distribution $P$ is given by $f(s):=f(x_1, x_2)=\frac 1{3}$ for all $(x_1, x_2)\in L$, and zero otherwise. Notice that $L=L_1\uu L_2\uu L_3$, where
\begin{align*}
L_1&=\set{(x_1, x_2) :  x_1=t, \, x_2=0 \te{ for } 0\leq t\leq 1},\\
L_2&=\set{(x_1, x_2) : x_1=t, \, x_2=-\sqrt 3(t-1) \te{ for } \frac 12 \leq t\leq 1}, \\
L_3&=\set{(x_1, x_2) : x_1=t, \, x_2=\sqrt 3 t \te{ for } 0\leq t\leq \frac 12}.
\end{align*}
Again, $dP(s)=P(ds)=f(x_1, x_2) ds=\frac 1 3 ds$. On $L_1$, we have $ds=\sqrt{(\frac {dx_1}{dt})^2 +(\frac {dx_2}{dt})^2}|dt|=|dt|$ yielding
$dP=\frac 1 3  |dt|$. On $L_2$ and $L_3$, we have $ds=\sqrt{(\frac {dx_1}{dt})^2 +(\frac {dx_2}{dt})^2}|dt|=2|dt|$ yielding $dP=\frac 23 |dt|$.

Let us now prove the following lemma.
\begin{lemma} \label{lemma00}
Let $X$ be a continuous random variable with uniform distribution taking values on $L$. Then,
$E(X)=(\frac 1 2, \frac {\sqrt 3}{ 6} ) \te{ and } V(X)=\frac 1 6.$
\end{lemma}
\begin{proof} Recall that by $(a, b)$ it is meant $a i+b j$, where $i$ and $j$ are the two unit vectors in the positive directions of $x_1$- and $x_2$-axes, respectively. Thus, we have,
\begin{align*}
E(X) =\int_L(x_1 i+x_2 j) dP=\frac{1}{3} \int_0^1 (t,0) \, dt-\frac{2}{3} \int_1^{\frac{1}{2}} (t,-\sqrt{3} (t-1)) \, dt-\frac{2}{3} \int_{\frac{1}{2}}^0 (t,\sqrt{3} t) \, dt=(\frac 1 2, \frac {\sqrt 3}{ 6} ),
\end{align*}
and
\begin{align*} V(X)&=E\|X-E(X)\|^2= \frac 1 {3} \Big(\int_{0}^1 \Big((t-\frac 1 2)^2 +(0-\frac {\sqrt 3}{ 6})^2\Big) dt-2 \int_{1}^{\frac 12} \Big((t-\frac 1 2)^2 \\
&+(-\sqrt 3 (t-1)-\frac {\sqrt 3}{ 6})^2\Big)  dt  -2 \int_{\frac 1 2}^{0} \Big((t-\frac 1 2)^2 +(\sqrt 3 t-\frac {\sqrt 3}{ 6})^2\Big)  dt\Big)=\frac 16.
\end{align*}
Hence, the proof of the lemma is complete.
\end{proof}

\begin{remark} For any $(a, b) \in \D R^2$, we have
\begin{align*}
&E\|X-(a, b)\|^2=\int_L\|(x_1, x_2)-(a, b)\|^2dP=\frac 1 3 \int_L \Big((x_1-a)^2+(x_2-b)^2\Big)ds\\
&=\frac 1 3 \int_L \Big((x_1-\frac 12 )^2+(x_2-\frac {\sqrt 3}{ 6} )^2\Big)ds+\frac 13\Big((\frac 12 -a)^2+(\frac {\sqrt 3}{ 6} -b)^2\Big)\\
&=V(X)+\frac 1 3 \|(a, b)-(\frac 12, \frac {\sqrt 3}{ 6})\|^2,
\end{align*}
which is minimum if $(a, b)=(\frac 12, \frac {\sqrt 3}{ 6})$, and the minimum value is $V(X)$.
Thus, we see that the optimal set of one-mean is the set $\set{(\frac 12, \frac {\sqrt 3}{ 6})}$, and the corresponding quantization error is the variance $V:=V(X)$ of the random variable $X$.
\end{remark}

The following lemma gives the optimal set of two-means.

\begin{lemma}  \label{lemma2} Let $P$ be the uniform distribution defined on the boundary of the equilateral triangle with vertices $O(0, 0)$, $A(1, 0)$, and $B(\frac 12, \frac {\sqrt 3}{2})$. Then, if we divide the equilateral triangle into an isosceles trapezoid and an equilateral triangle in the ratio $\ga : 1-\ga$, where $\ga=\frac{1}{8} (\sqrt{17}-1)$, then the conditional expectations of the two regions give an optimal set of two-means. One of the three such sets is given by $\set{(0.314187, 0.395954), (0.771396, 0.131985)}$, and the quantization error for two-means is $V_2=0.0994281$.
\end{lemma}
\begin{proof}
 Let the points $P$ and $Q$ with position vectors $\tilde p$ and $\tilde q$ form an optimal set of two-means for the uniform distribution on $L$. Let $\ell$ be the boundary of their Voronoi regions. Let $\ell$ intersect $OA$ and $AB$ at the two points $D$ and $E$ with position vectors $\tilde d$ and $\tilde e$, respectively. Let $D$ and $E$ are given by the parametric values $t=\ga$ and $t=\gb$. Then, $\tilde d=(\ga, 0)$, and $\tilde e=(\gb, -\sqrt 3(\gb-1))$. Let $P$ be in the region which contains the point $A$, and $Q$ be in the region which contains the point $B$.
 Then, \[\tilde p=E(X : X\in DA\uu AE), \te{ and } \tilde q=E(X : X\in EB\uu BO\uu OD),\]
 yielding
 \begin{align*}
 \tilde p=\frac 1{P(DA\uu AE)}\mathop{\int}\limits_{DA\uu AE} x dP=\Big(\frac{-\frac{\alpha ^2}{2}-2 (\frac{\beta ^2}{2}-\frac{1}{2})+\frac{1}{2}}{-\alpha -2 (\beta -1)+1},-\frac{2 (-\frac{1}{2} \sqrt{3} \beta ^2+\sqrt{3} \beta -\frac{\sqrt{3}}{2})}{-\alpha -2 (\beta -1)+1}\Big),
 \end{align*}
and
\[\tilde q=\frac 1{P(EB\uu BO\uu OD)}\mathop{\int}\limits_{EB\uu BO\uu OD} x dP=\Big(\frac{\frac{\alpha ^2}{2}-2 (\frac{1}{8}-\frac{\beta ^2}{2})+\frac{1}{4}}{\alpha -2 (\frac{1}{2}-\beta )+1},\frac{\frac{\sqrt{3}}{4}-2 (\frac{\sqrt{3} \beta ^2}{2}-\sqrt{3} \beta +\frac{3 \sqrt{3}}{8})}{\alpha -2 (\frac{1}{2}-\beta )+1}\Big).\]
Since $D$ and $E$ lie on $\ell$, and the boundary of the Voronoi regions $\ell$ is the perpendicular bisector of the line segment joining $P$ and $Q$, we have the canonical equations as $\rho(\tilde d, \tilde p)=\rho(\tilde d, \tilde q)$ and $\rho(\tilde e, \tilde p)=\rho(\tilde e, \tilde q)$ which after simplification give the following two equations:
\begin{align*}
&2 \alpha ^5+\alpha ^4 (12 \beta -11)+4 \alpha ^3 \left(4 \beta ^2-10 \beta +3\right)+2 \alpha ^2 \left(8 \beta ^3-26 \beta ^2+28 \beta -3\right)\\
&\qquad \qquad +2 \alpha  \left(8 \beta ^3-20 \beta ^2+12 \beta -3\right)-64 \beta ^5+240 \beta ^4-352 \beta ^3+252 \beta ^2-84 \beta +9=0,\notag
\end{align*}
and
\begin{align*}
&-9-5 \alpha ^4+2 \alpha ^5+12 \beta +108 \beta ^2-288 \beta ^3+240 \beta ^4-64 \beta ^5+\alpha ^3 \left(-8 \beta ^2+4 \beta +4\right)\\
&\qquad \qquad +\alpha ^2 \left(-32 \beta ^3+76 \beta ^2-44 \beta +6\right)-2 \alpha  \left(48 \beta ^4-144 \beta ^3+160 \beta ^2-78 \beta +15\right)=0. \notag
\end{align*}
Solving the above two equations in $\ga$ and $\gb$, we get the following four sets of solutions:
\begin{align*} & \Big\{\alpha=0,\beta=\frac{3}{4}\Big\},\Big\{\alpha=\frac{1}{2},\beta=\frac{1}{2}\Big\},\{\alpha=1,\beta=1\},\\
&\qquad \te{ and }  \Big\{\alpha=\frac{1}{8} (\sqrt{17}-1),\beta=\frac{\frac{79}{8} (\sqrt{17}-1)-51}{65 (\sqrt{17}-1)-232}\Big\}
\end{align*}
Thus, we have the following results:

$(i)$
Putting $\{\alpha=0,\beta=\frac{3}{4}\}$, we obtain $\tilde p=(\frac{5}{8},\frac{1}{8 \sqrt{3}})$ and $\tilde q=(\frac{3}{8},\frac{7}{8 \sqrt{3}})$. Notice that in this case $D$ coincides with the vertex $O$, and $\tilde p$ and $\tilde q$ are symmetric about the line $\ell$, i.e., the boundary line $\ell$ of the two Voronoi regions becomes a median of the triangle passing through the vertex $O$, and thus, the corresponding distortion error, denoted by $V_{2, (i)}$, is given by
\begin{align*}
V_{2,(i)}&=2\int_{OA\uu AE}\rho((x_1, x_2), \tilde p) dP=\frac 2 3\Big(\int_\ga^1\rho((t, 0), \tilde p) dt-2\int_1^\gb\rho((t, -\sqrt 3(t-1)), \tilde p) dt\Big)
\end{align*}
yielding $V_{2,(i)}=\frac{5}{48}=0.104167$.

$(ii)$
Putting $\Big\{\alpha=\frac{1}{2},\beta=\frac{1}{2}\Big\}$, we obtain $\tilde p=(\frac{3}{4},\frac{1}{2 \sqrt{3}})$ and $\tilde q=(\frac{1}{4},\frac{1}{2 \sqrt{3}})$. Notice that in this case $E$ coincides with the vertex $B$, and $\tilde p$ and $\tilde q$ are symmetric about the line $\ell$, i.e., the boundary line $\ell$ of the two Voronoi regions becomes a median of the triangle passing through the vertex $B$, and thus, the corresponding distortion error, denoted by $V_{2, (ii)}$, is given by
\begin{align*}
V_{2,(ii)}&=2\int_{DA\uu AB}\rho((x_1, x_2), \tilde p) dP=\frac 2 3\Big(\int_\ga^1\rho((t, 0), \tilde p) dt-2\int_1^\gb\rho((t, -\sqrt 3(t-1)), \tilde p) dt\Big)
\end{align*}
yielding $V_{2,(ii)}=\frac{5}{48}=0.104167$.

$(iii)$ Putting $\{\alpha=1,\beta=1\}$, we see that the points $D$ and $E$ coincides with $A$, which is not true.

$(iv)$ Putting $ \{\alpha=\frac{1}{8} (\sqrt{17}-1),\beta=\frac{\frac{79}{8} (\sqrt{17}-1)-51}{65 (\sqrt{17}-1)-232} \} $, we see that the points $D$ and $E$, respectively, are given by $(\frac{1}{8} (\sqrt{17}-1), 0)$ and $\Big(\frac{\frac{79}{8} (\sqrt{17}-1)-51}{65 (\sqrt{17}-1)-232},-\sqrt{3} \Big(\frac{\frac{79}{8} (\sqrt{17}-1)-51}{65 (\sqrt{17}-1)-232}-1\Big)\Big)$. Thus, in this case we have $OD=BE$ and $DA=AE=ED$ implying the fact that the boundary of the Voronoi regions $\ell$ divides the equilateral triangle into an isosceles trapezoid and an equilateral triangle in the ratio $\ga : 1-\ga$, where $\ga=\frac{1}{8} (\sqrt{17}-1)$. The corresponding distortion error $V_{2, (iv)}$ is given by
\begin{align*} V_{2,(iv)}&=\frac 1 3\Big(\int_{\alpha }^1 \rho((t, 0), \tilde p) \, dt- 2\int_1^{\beta } \rho((t,-\sqrt{3} (t-1)), \tilde p) \, dt -2\int_{\beta }^{\frac{1}{2}} \rho((t,-\sqrt{3} (t-1)), \tilde q) \, dt \\
&\qquad-2\int_{\frac{1}{2}}^0  \rho((t,\sqrt{3} t), \tilde q) \, dt+\int_0^{\alpha } \rho((t, 0), \tilde q) \, dt\Big)\\
&=\frac 1 3(0.0330382+0.0330382+0.0716907+0.0888266+0.0716907)=0.0994281.
\end{align*}
Thus, comparing the distortion errors $V_{2,(i)}$, $V_{2,(ii)}$, and $V_{2,(iv)}$, we can say that when the boundary of the Voronoi regions $\ell$ divides the equilateral triangle into an isosceles trapezoid and an equilateral triangle in the ratio $\ga : 1-\ga$, where $\ga=\frac{1}{8} (\sqrt{17}-1)$, then the conditional expectations of the two Voronoi regions give an optimal set of two-means. There are three such sets; one of them is given by $\set{(0.314187, 0.395954), (0.771396, 0.131985)}$ and the quantization error for two-means is $V_2=0.0994281$.
Thus, the proof of the lemma is complete.
\end{proof}
The following lemma gives the optimal set of three-means.

\begin{lemma}  \label{lemma14}  Let $P$ be the uniform distribution defined on the boundary of the equilateral triangle with vertices $O(0, 0)$, $A(1, 0)$, and $B(\frac 12, \frac {\sqrt 3}{2})$. Then, the set $\set{(\frac{13}{16},\frac{\sqrt{3}}{16}), (\frac{1}{2},\frac{3 \sqrt{3}}{8}), (\frac{3}{16},\frac{\sqrt{3}}{16})}$ forms the optimal set of three-means, and the quantization error for three-means is $V_3=\frac{7}{192}=0.0364583$.
\end{lemma}

\begin{proof}
Let the points $P$, $Q$, and $R$ with position vectors $\tilde p$, $\tilde q$, and $\tilde r$ form an optimal set of three-means for the uniform distribution on $L$. Let the boundaries of their Voronoi regions cut the sides $OA$, $AB$, and $BO$ at the points $D$, $E$, and $F$ with parametric values given by $t=\ga$, $t=\gb$, and $t=\gg$, respectively. Let $\tilde d$, $\tilde e$, and $\tilde f$ be the position vectors of $D$, $E$, and $F$, respectively. Then,
\[ \tilde d=(\ga, 0), \, \tilde e=(\beta , -\sqrt{3} (\beta -1)),  \, \te{ and }  \tilde f=(\gamma , \sqrt{3} \gamma).\]
Let $\tilde p=E(X : X \in DA\uu AE)$, $\tilde q=E(X : X \in EB\uu BF)$, and $\tilde r=E(X : X\in FO\uu OD)$. Then, using the definitions of conditional expectations, and after some calculations, we have
\begin{align*}
\tilde p&=\Big(\frac{-\frac{\alpha ^2}{2}-2  (\frac{\beta ^2}{2}-\frac{1}{2})+\frac{1}{2}}{-\alpha -2 (\beta -1)+1},-\frac{2 (-\frac{1}{2} \sqrt{3} \beta ^2+\sqrt{3} \beta -\frac{\sqrt{3}}{2})}{-\alpha -2 (\beta -1)+1}\Big), \\
 \tilde q& =\Big(\frac{-2(\frac{1}{8}-\frac{\beta ^2}{2})-2 (\frac{\gamma ^2}{2}-\frac{1}{8})}{-2 (\frac{1}{2}-\beta)-2 (\gamma -\frac{1}{2})},\frac{-2 (\frac{\sqrt{3} \beta ^2}{2}-\sqrt{3} \beta +\frac{3 \sqrt{3}}{8})-2 (\frac{\sqrt{3} \gamma ^2}{2}-\frac{\sqrt{3}}{8})}{-2 (\frac{1}{2}-\beta)-2 (\gamma -\frac{1}{2})}\Big), \te{ and }\\
\tilde r&=\Big(\frac{\frac{\alpha ^2}{2}+\gamma ^2}{\alpha +2 \gamma },\frac{\sqrt{3} \gamma ^2}{\alpha +2 \gamma }\Big).
\end{align*}
Notice that the point $D$ lies on the boundary of the Voronoi regions of $P$ and $R$ implying $\rho(\tilde d, \tilde r)=\rho(\tilde d, \tilde p)$. Similarly, for the points $E$ and $F$, we have $\rho(\tilde e, \tilde p)=\rho(\tilde e, \tilde q)$, and $\rho(\tilde f, \tilde q)=\rho(\tilde f, \tilde r)$. Thus, we have three equations in $\ga$, $\gb$, and $\gg$. Solving the three equations, we obtain the following three sets of solutions:
\[ \{\alpha=0,\beta=1,\gamma=\frac{1}{2} \}, \{\alpha=\frac{1}{2},\beta=\frac{3}{4},\gamma=\frac{1}{4} \}, \te { and }  \{\alpha=1,\beta=\frac{1}{2},\gamma=0 \}.\]
Thus, we obtain the following three results:

$(i)$ For $ \{\alpha=0,\beta=1,\gamma=\frac{1}{2} \}$, we have $\tilde p=(\frac{1}{2}, 0), \, \tilde q=(\frac{3}{4},\frac{\sqrt{3}}{4}), \, \tilde r=(\frac{1}{4},\frac{\sqrt{3}}{4})$ with distortion error obtained as $\frac 1{12}=0.0833333$.

$(ii)$ For $ \{\alpha=\frac{1}{2},\beta=\frac{3}{4},\gamma=\frac{1}{4} \}$, we have $\tilde p=(\frac{13}{16},\frac{\sqrt{3}}{16}), \, \tilde q=(\frac{1}{2},\frac{3 \sqrt{3}}{8}), \, \tilde r=(\frac{3}{16},\frac{\sqrt{3}}{16})$ with distortion error obtained as $\frac{7}{192}=0.0364583$.

$(iii)$ For $ \{\alpha=1,\beta=\frac{1}{2},\gamma=0 \}$, we have $\tilde p=(\frac{3}{4},\frac{\sqrt{3}}{4}), \, \tilde q=(\frac{1}{4},\frac{\sqrt{3}}{4}), \, \tilde r=(\frac{1}{2},0)$ with distortion error obtained as $\frac 1{12}=0.0833333$.

Comparing the distortion errors, we deduce that the set $\set{(\frac{13}{16},\frac{\sqrt{3}}{16}), (\frac{1}{2},\frac{3 \sqrt{3}}{8}), (\frac{3}{16},\frac{\sqrt{3}}{16})}$ forms the optimal set of three-means (see Figure~\ref{Fig1}), and the quantization error for three-means is $V_3=\frac{7}{192}=0.0364583$. Thus, the lemma is yielded.
\end{proof}

The following lemma gives the optimal set of four-means.

\begin{lemma} \label{lemma15}Let $P$ be the uniform distribution defined on the boundary of the equilateral triangle with vertices $O(0, 0)$, $A(1, 0)$, and $B(\frac 12, \frac {\sqrt 3}{2})$. Then, the set
\[\set{(0.133784, 0.140735), (0.5, 0), (0.866216, 0.140735), (0.5, 0.653763)}\] forms an optimal set of four-means, and the corresponding quantization error is $V_4=0.028269$.
\end{lemma}

\begin{proof}
Let the points $P$, $Q$, $R$, and $S$ with position vectors $\tilde p$, $\tilde q$, $\tilde r$, and $\tilde s$ form an optimal set of four-means for the uniform distribution on $L$. Let the boundaries of their Voronoi regions cut the side $OA$ at two points $D$ and $E$, and cut $AB$, $BO$ at the points $F$, $G$, respectively. Let the position vectors of $D, E, F, G$ be $\tilde d, \tilde e, \tilde f, \tilde g$ with parametric values given by $t=\ga, \gb, \gg, \gd$, respectively.   Then,
\[ \tilde d=(\ga, 0), \, \tilde e=(\gb, 0), \, \tilde f=(\gg , -\sqrt{3} (\gg -1)),  \, \te{ and }  \tilde g=(\gd, \sqrt{3} \gd).\]
Let $\tilde p=E(X : X \in GO\uu OD)$, $\tilde q=E(X : X \in DE)$, $\tilde r=E(X : X\in EA\uu AF)$, and $\tilde s=E(X : X \in FB\uu BG)$. Then, using the definitions of conditional expectations, and after some calculations, we have
\begin{align*}
\tilde p&=\Big(\frac{\frac{\alpha ^2}{2}+\delta ^2}{\alpha +2 \delta },\frac{\sqrt{3} \delta ^2}{\alpha +2 \delta }\Big), \ \tilde q=\Big(\frac{\alpha +\beta }{2},0\Big),\\
\tilde r&=\Big(\frac{-\frac{\beta ^2}{2}-2  (\frac{\gamma ^2}{2}-\frac{1}{2} )+\frac{1}{2}}{-\beta -2 (\gamma -1)+1},-\frac{2  (-\frac{1}{2} \sqrt{3} \gamma ^2+\sqrt{3} \gamma -\frac{\sqrt{3}}{2} )}{-\beta -2 (\gamma -1)+1}\Big), \te{ and }\\
\tilde s&=\Big(\frac{-2  (\frac{1}{8}-\frac{\gamma ^2}{2} )-2  (\frac{\delta ^2}{2}-\frac{1}{8} )}{-2  (\frac{1}{2}-\gamma  )-2  (\delta -\frac{1}{2} )},\frac{-2  (\frac{\sqrt{3} \gamma ^2}{2}-\sqrt{3} \gamma +\frac{3 \sqrt{3}}{8} )-2  (\frac{\sqrt{3} \delta ^2}{2}-\frac{\sqrt{3}}{8} )}{-2  (\frac{1}{2}-\gamma  )-2  (\delta -\frac{1}{2} )}\Big).
\end{align*}
In this case, we obtain the following four canonical equations in $\ga, \gb, \gg$, and $\gd$:
 \[\rho(\tilde d, \tilde p)=\rho(\tilde d, \tilde q), \ \rho(\tilde e, \tilde q)=\rho(\tilde e, \tilde r),\ \rho(\tilde f, \tilde r)=\rho(\tilde f, \tilde s), \te{ and } \rho(\tilde g, \tilde s)=\rho(\tilde g, \tilde p).  \]
 Solving the above four equations, we obtain
 \[ \alpha= 0.28985, \ \beta = 0.71015, \ \gamma = 0.7451, \ \delta = 0.2549 \]
 yielding $ \tilde p= (0.133784, 0.140735), \, \tilde q= (0.5, 0), \, \tilde r =(0.866216, 0.140735), \te{ and } \tilde s=(0.5, 0.653763)$ (see Figure~\ref{Fig1}). Notice that the above four points are symmetric about the median passing through the vertex $B$. Thus, the corresponding quantization error is given by
 \begin{align*}
 V_4&=\frac 2 3 \Big(\int_{\frac 12}^\gb  \rho((t, 0), \tilde q) dt+\int_{\gb}^1  \rho((t, 0), \tilde r) dt-2\int_{1}^\gg  \rho((t, -\sqrt{3} (t-1)), \tilde r) dt\\
 &\qquad \qquad -2\int_{\gg}^{\frac 12}  \rho((t, -\sqrt{3} (t-1)), \tilde s) dt\Big)=0.028269.
 \end{align*}
 Thus, the lemma is yielded.
\end{proof}

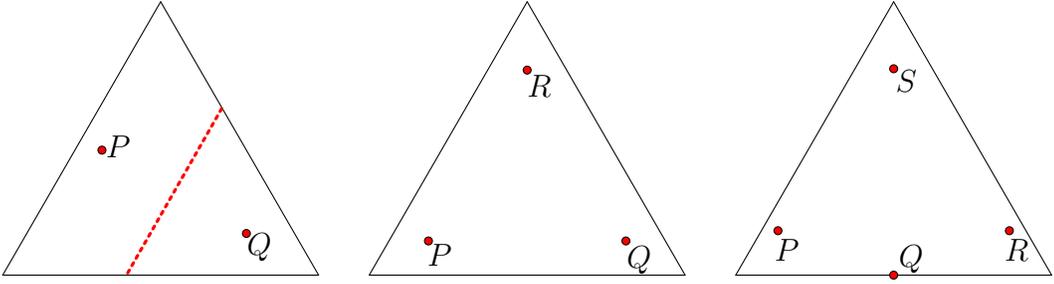
\begin{figure}
\begin{tikzpicture}[line cap=round,line join=round,>=triangle 45,x=0.7cm,y=0.7cm]
\clip(-0.3068784005849317,-0.3327561393976855) rectangle (6.457388811335072,5.386798747680555);
\draw (0.,0.)-- (6.,0.);
\draw (0.,0.)-- (3.,5.196152422706632);
\draw (6.,0.)-- (3.,5.196152422706632);
\draw [line width=1.2pt,dotted,color=ffqqqq] (4.1580549682875265,3.1634556986354005)-- (2.331416490486258,-0.0331616375168149);
\draw (1.759323467230444,2.8590159523351897) node[anchor=north west] {$P$};
\draw (4.431627906976744,0.9985508360561224) node[anchor=north west] {$Q$};
\begin{scriptsize}
\draw [fill=ffqqqq] (4.62837,0.791909) circle (1.5pt);
\draw [fill=ffqqqq] (1.88512,2.37573) circle (1.5pt);
\end{scriptsize}
\end{tikzpicture}
\begin{tikzpicture}[line cap=round,line join=round,>=triangle 45,x=0.7cm,y=0.7cm]
\clip(-0.3068784005849317,-0.3327561393976855) rectangle (6.457388811335072,5.386798747680555);
\draw (0.,0.)-- (6.,0.);
\draw (0.,0.)-- (3.,5.196152422706632);
\draw (6.,0.)-- (3.,5.196152422706632);
\draw (4.685000044513341,0.7641114075321921) node[anchor=north west] {$Q$};
\draw (0.9013843799842402,0.801847564177483) node[anchor=north west] {$P$};
\draw (2.798192212248791,4.00942087902722) node[anchor=north west] {$R$};
\begin{scriptsize}
\draw [fill=ffqqqq] (4.875,0.649519052838329) circle (1.5pt);
\draw [fill=ffqqqq] (3.,3.897114317029974) circle (1.5pt);
\draw [fill=ffqqqq] (1.125,0.649519052838329) circle (1.5pt);
\end{scriptsize}
\end{tikzpicture}
\begin{tikzpicture}[line cap=round,line join=round,>=triangle 45,x=0.7cm,y=0.7cm]
\clip(-0.3068784005849317,-0.3327561393976855) rectangle (6.257388811335072,5.386798747680555);
\draw (0.,0.)-- (6.,0.);
\draw (0.,0.)-- (3.,5.196152422706632);
\draw (6.,0.)-- (3.,5.196152422706632);
\draw (2.890770931130555,0.7777541823257719) node[anchor=north west] {$Q$};
\draw (0.5346525645854395,0.890843222930993) node[anchor=north west] {$P$};
\draw (2.835928368894082,4.097471911199566) node[anchor=north west] {$S$};
\draw (4.899468016557433,0.890843222930993) node[anchor=north west] {$R$};
\begin{scriptsize}
\draw [fill=ffqqqq] (0.8027039999999999,0.84441) circle (1.5pt);
\draw [fill=ffqqqq] (3.,0.) circle (1.5pt);
\draw [fill=ffqqqq] (5.197296,0.84441) circle (1.5pt);
\draw [fill=ffqqqq] (3.,3.9225779999999997) circle (1.5pt);
\end{scriptsize}
\end{tikzpicture}
 \caption{Optimal configuration of $n$ means for $2\leq n\leq 4$.} \label{Fig1}
\end{figure}
\begin{remark}
Due to the symmetry of the equilateral triangle with respect to a rotation of $\frac {2\pi} 3$, by Lemma~\ref{lemma15}, we conclude that there are three optimal sets of four-means (see Figure~\ref{Fig1}).
\end{remark}

Let us now state the following two lemmas which give the optimal sets of five- and six-means. The proofs are similar to the previous lemmas.

\begin{lemma}\label{lemma16}
Let $P$ be the uniform distribution defined on the boundary of the equilateral triangle with vertices $O(0, 0)$, $A(1, 0)$, and $B(\frac 12, \frac {\sqrt 3}{2})$. Then, the set
\begin{align*} \set{ (0.130625, 0.138564), (0.485912, 0), &(0.883966, 0.0669921), (0.742956, 0.445213),  \\
&\qquad \qquad(0.445312, 0.683619)}
\end{align*} forms an optimal set of five-means, and the corresponding quantization error is $V_5=0.020525$ (see Figure~\ref{Fig2}).
\end{lemma}

\begin{figure}
\begin{tikzpicture}[line cap=round,line join=round,>=triangle 45,x=0.7cm,y=0.7cm]
\clip(-0.1127819138037,-0.3134389727477597) rectangle (6.227037770644073,5.331444573458087);
\draw (0.,0.)-- (6.,0.);
\draw (0.,0.)-- (3.,5.196152422706632);
\draw (6.,0.)-- (3.,5.196152422706632);
\draw (0.5872932176763886,1.0131876237961868) node[anchor=north west] {$P$};
\draw (2.7759483641827734,0.6872672631650832) node[anchor=north west] {$Q$};
\draw (5.110208319103884,0.5774250392732724) node[anchor=north west] {$R$};
\draw (3.798262413769549,2.8463957379960334) node[anchor=north west] {$S$};
\draw (2.6407117000204897,4.288920155727061) node[anchor=north west] {$T$};
\begin{scriptsize}
\draw [fill=ffqqqq] (2.9154720000000003,0.) circle (1.5pt);
\draw [fill=ffqqqq] (5.303796,0.4019526) circle (1.5pt);
\draw [fill=ffqqqq] (4.457736,2.671278) circle (1.5pt);
\draw [fill=ffqqqq] (2.671872,4.101713999999999) circle (1.5pt);
\draw [fill=ffqqqq] (0.78375,0.8313839999999999) circle (1.5pt);
\end{scriptsize}
\end{tikzpicture} \quad
\begin{tikzpicture}[line cap=round,line join=round,>=triangle 45,x=0.7cm,y=0.7cm]
\clip(-0.1127819138037,-0.373438972747758) rectangle (6.227037770644073,5.331444573458088);
\draw (0.,0.)-- (6.,0.);
\draw (0.,0.)-- (3.,5.196152422706632);
\draw (6.,0.)-- (3.,5.196152422706632);
\draw (0.5872932176763886,1.0131876237961868) node[anchor=north west] {$P$};
\draw (2.7759483641827734,0.6872672631650832) node[anchor=north west] {$Q$};
\draw (5.110208319103884,0.5774250392732724) node[anchor=north west] {$R$};
\draw (3.798262413769549,2.8463957379960334) node[anchor=north west] {$S$};
\draw (2.996158732327026,4.589446076087692) node[anchor=north west] {$T$};
\draw (1.4235817225599344,2.7111590738337505) node[anchor=north west] {$U$};
\begin{scriptsize}
\draw [fill=ffqqqq] (0.677124,0.3909378) circle (1.5pt);
\draw [fill=ffqqqq] (4.5,2.598076211353316) circle (1.5pt);
\draw [fill=ffqqqq] (3.,4.4142779999999995) circle (1.5pt);
\draw [fill=ffqqqq] (5.322876,0.3909378) circle (1.5pt);
\draw [fill=ffqqqq] (1.5,2.598078) circle (1.5pt);
\draw [fill=ffqqqq] (3.,0.) circle (1.5pt);
\end{scriptsize}
\end{tikzpicture}
 \caption{Optimal configuration of $n$ means for $n=5, \, 6$.} \label{Fig2}
\end{figure}
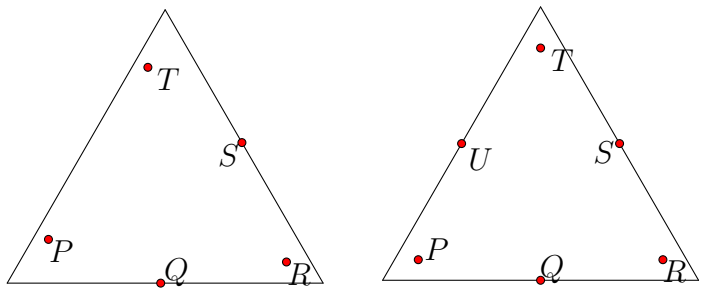

\begin{lemma} \label{lemma17} Let $P$ be the uniform distribution defined on the boundary of the equilateral triangle with vertices $O(0, 0)$, $A(1, 0)$, and $B(\frac 12, \frac {\sqrt 3}{2})$. Then, the set
\begin{align*}
\set{(0.112854, 0.0651563), (0.5, 0), & (0.887146, 0.0651563), (0.75, 0.433013), \\
&(0.5, 0.735713), (0.25, 0.433013)}
\end{align*} forms an optimal set of six-means, and the corresponding quantization error is $V_6=0.0132077$ (see Figure~\ref{Fig2}).
\end{lemma}

\begin{remark} \label{remark101}
By Lemma~\ref{lemma14}, Lemma~\ref{lemma15}, Lemma~\ref{lemma16}, and Lemma~\ref{lemma17}, we see that if $\ga_n$ is an optimal set of $n$-means for $3\leq n\leq 6$, then $\ga_n$ contains three elements in the interior of the triangle each close to one of
the vertices, and the remaining elements are on the sides of the triangle.
In fact, for $n\geq 6$, if $\ga_n$ is an optimal set of $n$-means, then it can be shown that $\ga_n$ contains three elements in the interior of the triangle each close to one of
the vertices, and the remaining elements are on the sides of the triangle. If $\ga_n$ contains $n_1$, $n_2$, and $n_3$ elements from the three sides of the triangle, where $n=n_1+n_2+n_3+3$, it can be shown that  $|n_i-n_j|\leq 1$ for $1\leq i\neq j\leq 3$. If it is not true, then the quantization error can be reduced by redistributing the points until $n_1$, $n_2$, and $n_3$ satisfy $|n_i-n_j|\leq 1$ for $1\leq i\neq j\leq 3$, and this contradicts the optimality of the set $\ga_n$. Due to technicality, we are not showing the details of the proof in the paper.
\end{remark}

The following theorem determines the optimal sets of $n$-means and the $n$th quantization errors when $n=3k+3$ for some positive integer $k$. It also helps us to determine the quantization dimension and the quantization coefficient for the uniform distribution defined on the boundary of the equilateral triangle.

\begin{theorem} \label{Th41} Let $n\in \D N$ be such that $n\geq 6$, and $n=3k+3$ for some positive integer $k$. Then, the optimal set of $n$-means for $P$ is given by
\begin{align*} \ga_n=\Big\{(\frac{3 r }{8},\frac{\sqrt{3} r }{8}), \, (1-\frac{3 r }{8},\frac{\sqrt{3} r }{8}), \,  (\frac{1}{2},-\frac{1}{4} \sqrt{3} (r -2)) \Big\} \uu \gg \uu T_1(\gg)\uu T_2(\gg),
\end{align*}
where $\gg:=\set{r+\frac{2j-1}{2k}(1-2r) : j=1, 2, \cdots, k}$, and $r=\frac{8-2 \sqrt{7} k}{16-7 k^2}$, and  $T_1$, $T_2$ are two affine transformations on $\D R^2$, such that $T_1(0,0)=(\frac 12, \frac {\sqrt{3}} 2),\, T_1(1, 0)=(1, 0), \, T_2(0, 0)=(0, 0), \te{ and } T_2(1, 0)=(\frac 12, \frac {\sqrt{3}} 2)$.
The quantization error for $n$-means is given by
\[V_n=\frac{7 \left(7 k^2-8 \sqrt{7} k+16\right)}{12 \left(16-7 k^2\right)^2}.\]

\end{theorem}

\begin{proof}
Let $a, b, c$ be the three points that $\ga_n$ contains from the interior of the angles $\angle O$, $\angle A$, and $\angle B$, respectively. Again, recall that $P$ is uniform over the boundary of the triangle, and as mentioned in Remark~\ref{remark101}, $\ga_n$ contains $k$ elements from each side of the triangle yielding the fact that the Voronoi regions of $a, b, c$ will form equilateral triangles $P$ almost surely. Let the lengths of the sides of the equilateral triangles formed by the Voronoi regions of $a, b, c$ equal $r$. Then,
\begin{align*} a&=\frac{\int_0^{r }(t, 0) \, dt-2 \int_{\frac{r }{2}}^0 (t,\sqrt{3} t) \, dt}{\int_0^{r } 1 \, dt-2 \int_{\frac{r }{2}}^0 1 \, dt}=(\frac{3 r }{8},\frac{\sqrt{3} r }{8}),\\
b&=\frac{\int_{1-r}^1 (t,0) \, dt-2 \int_1^{1-\frac{r}{2}} (t,-\sqrt{3} (t-1))\, dt}{\int_{1-r}^1 1 \, dt-2 \int_1^{1-\frac{r}{2}} 1 \, dt}=(1-\frac{3 r}{8},\frac{\sqrt{3} r}{8}), \te{ and }\\
c&=\frac{-2 \int_{\frac 12(1+r)}^{\frac{1}{2}} (t,-\sqrt{3} (t-1)) \, dt-2 \int_{\frac{1}{2}}^{\frac 12(1-r)} (t,\sqrt{3} t) \, dt}{-2 \int_{\frac 12(1+r)}^{\frac{1}{2}} 1 \, dt-2 \int_{\frac{1}{2}}^{\frac 12(1-r)} 1\, dt}=(\frac{1}{2},-\frac{1}{4} \sqrt{3} (r-2)).
\end{align*}
Let $\gg$ be the set of all the $k$ points that $\ga_n$ contains from the side $OA$. Then, the Voronoi regions of the points in $\gg$ covers the closed interval $[r, 1-r]$ yielding
\[\gg=\Big\{\Big(r+\frac{2j-1}{2k}(1-2r), 0\Big) : j=1, 2, \cdots, k\Big\}.\]
Since $T_1$ and $T_2$ are affine transformations with $T_1(OA)=AB$, and $T_2(OA)=OB$, we have the points that $\ga_n$ contains from $AB$ and $OB$ are, respectively, $T_1(\gg)$ and $T_2(\gg)$ yielding
\begin{align*} \ga_n=\Big\{(\frac{3 r }{8},\frac{\sqrt{3} r }{8}), \, (1-\frac{3 r }{8},\frac{\sqrt{3} r }{8}), \,  (\frac{1}{2},-\frac{1}{4} \sqrt{3} (r -2)) \Big\} \uu \gg \uu T_1(\gg)\uu T_2(\gg).
\end{align*}
Using the symmetry, the quantization error for $n$-means is obtained as
\begin{align*}
V_n&=3(\te{quantization error due to the points  $a$ and the $k$ points in $\gg$})\\
&=3\Big(-\frac 2 3\int_{\frac {r}{2}}^0\rho((t, \sqrt{3}t), a) dt+\frac{1}{3} \int_0^{r } \rho((t, 0), a) dt\\
&\qquad \qquad +\frac{1}{3} k \Big(\int_{r }^{r +\frac{1-2 r }{k}} \rho\Big((t, 0), (r+\frac{1}{2k}(1-2r), 0)\Big) dt\Big)\\
&=\frac{1}{24} \Big(7 r ^3-\frac{2 (2 r -1)^3}{k^2}\Big).
\end{align*}
Notice that for a given $k$, the quantization error $V_n$ is a function of $r$. Solving $\frac{\pa V_n}{\pa r}=0$, we have $r=\frac{8-2 \sqrt{7} k}{16-7 k^2}$. Putting $r=\frac{8-2 \sqrt{7} k}{16-7 k^2}$, we have
\[V_n=\frac{7 \left(7 k^2-8 \sqrt{7} k+16\right)}{12 \left(16-7 k^2\right)^2}.\]
Thus, the proof the theorem is complete.
\end{proof}

\begin{remark}  When $n=6$, i.e., when $k=1$, then Theorem~\ref{Th41} reduces to Lemma~\ref{lemma17}. Using a similar technique as in the proof of Theorem~\ref{Th41}, we can also determine the optimal sets of $n$-means and the $n$th quantization errors when $n=3k+3+1$, or $n=3k+3+2$ for some positive integer $k$.
\end{remark}
\begin{prop}
Quantization dimension $D(P)$ of the uniform distribution $P$ defined on the boundary of the equilateral triangle equals the dimension of the boundary of the triangle. Moreover, the quantization coefficient exists as a finite positive number which equals $\frac 34$.
\end{prop}
\begin{proof} For $n\in \D N$, $n\geq 6$, let $\ell(n)$ be the unique positive real number such that $3\ell(n)+3\leq n<3(\ell(n)+1)+3$. Then,
$V_{3(\ell(n)+1)+3}<V_n\leq V_{3\ell(n)+3}$ implying
\begin{align} \label{eq45}
\frac {2 \log (3\ell(n)+3)}{-\log V_{3(\ell(n)+1)+3}}<\frac {2\log n}{-\log V_n} <\frac{2 \log (3(\ell(n)+1)+3)}{-\log V_{3\ell(n)+3}}.
\end{align}
Notice that
\[\lim_{n\to \infty} \frac {2 \log (3\ell(n)+3)}{-\log V_{3(\ell(n)+1)+3}}=\underset{\ell(n)\to \infty }{\text{lim}}\frac{2 \log (3 \ell(n)+3)}{-\log \left(\frac{7 \left(7 (\ell(n)+1)^2-8 \sqrt{7} (\ell(n)+1)+16\right)}{12 \left(16-7 (\ell(n)+1)^2\right)^2}\right)}=1,\]
and \[\lim_{n\to \infty} \frac{2 \log (3(\ell(n)+1)+3)}{-\log V_{3\ell(n)+3}}=\underset{\ell(n)\to \infty }{\text{lim}}\frac{2 \log (3 (\ell(n)+1)+3)}{-\log \left(\frac{7 \left(7 \ell(n)^2-8 \sqrt{7} \ell(n)+16\right)}{12 \left(16-7 \ell(n)^2\right)^2}\right)}=1\]
and hence, by \eqref{eq45}, $\mathop{\lim}\limits_{n\to \infty} \frac {2\log n}{-\log V_n}=1$ which is the dimension of the underlying space. Again,
\begin{equation} \label{eq46} (3\ell(n)+3)^2V_{3(\ell(n)+1)+3}<n^2 V_n<(3(\ell(n)+1)+3)^2V_{3\ell(n)+3}.
\end{equation}
We have
\[\lim_{n \to \infty} (3\ell(n)+3)^2V_{3(\ell(n)+1)+3}=\underset{\ell(n)\to \infty }{\text{lim}}(3\ell(n)+3)^2 \frac{7 \left(7 (\ell(n)+1)^2-8 \sqrt{7} (\ell(n)+1)+16\right)}{12 \left(16-7 (\ell(n)+1)^2\right)^2}=\frac 34,\]
and
\[\lim_{n \to \infty} (3(\ell(n)+1)+3)^2V_{3\ell(n)+3}=\underset{\ell(n)\to \infty }{\text{lim}}(3(\ell(n)+1)+3)^2 \frac{7 \left(7 \ell(n)^2-8 \sqrt{7} \ell(n)+16\right)}{12 \left(16-7 \ell(n)^2\right)^2}=\frac 34,\]
and hence, by \eqref{eq46}, we have
$\mathop{\lim}\limits_{n\to\infty} n^2 V_n=\frac 34$, i.e., the quantization coefficient exists as a finite positive number which equals $\frac 34$.
Thus, the proof of the proposition is complete.
\end{proof}

\begin{remark}
By Remark~\ref{remark101}, we see that an optimal set $\ga_n$ of $n$-means for $n\geq 3$ always contains
three elements in the interior of the triangle each close to one of
the vertices, and the remaining elements are on the sides of the triangle. It can be shown that as $n$ approaches to infinity each of the three elements which are in the interior of the triangle approaches to their closest vertices.
\end{remark}

Let us now end the paper with the following remark.

\begin{rem} \label{rem1}  We know that an affine transformation $T$ on $\D R^2$ is given by \[T\left[\begin{array}{cc}
x_1\\
x_2
\end{array}
\right]=A\left[\begin{array}{cc}
x_1\\
x_2
\end{array}
\right]+\left[\begin{array}{cc}
e\\
f
\end{array}
\right], \te{ where } A:=\left[\begin{array}{cc}
a & b\\
c  & d
\end{array}
\right]  \te{ is an invertible matrix, and } e, f \in \D R.\]
Once the optimal sets of $n$-means are known for a line segment, circle, or an equilateral triangle, by giving an affine transformation, one can easily obtain them for any line segment, circle, or an equilateral triangle. For example, by Theorem~\ref{theorem0}, we see that the optimal set of $n$-means with respect to the uniform distribution on the line segment $[0, 1]$ is given by $ \ga_n:=\set{\frac {2i-1}{2n} : 1\leq i\leq n}$, and the corresponding quantization error is
$V_n:=V_n(P)=\frac{1}{12 n^2}.$  The optimal set of $n$-means for the uniform distribution on the line segment joining $(0, 0)$ and $(1, \sqrt 3)$ is given by $ \ga_n:=\set{(\frac{2 i-1}{2 n}, \frac{\sqrt{3} (2 i-1)}{2 n}) : 1\leq i\leq n}$, and the corresponding quantization error is
$V_n:=V_n(P)=\frac{1}{3 n^2}$.
\end{rem}

\end{document}